\pdfoutput=1 
\documentclass[10pt,a4paper]{article}
\usepackage[utf8]{inputenc}
\usepackage{amsmath,amscd}
\usepackage{amsfonts}
\usepackage{amssymb}
\usepackage{amsthm}
\usepackage{xcolor}
\usepackage{enumitem}
\usepackage{tikz-cd}
\usepackage{pdfpages}
\usepackage{hyperref}

\DeclareMathOperator{\Ker}{Ker}

\DeclareMathOperator{\Aut}{Aut}
\DeclareMathOperator{\Abel}{Abel}

\newcommand{\PP}{\mathbb{P}}

\newcommand{\ZZ}{\mathbb{Z}}
\newcommand{\BB}{\mathbb{B}}
\newcommand{\CC}{\mathbb{C}}

\newtheorem{theorem}{Theorem}
\newtheorem{corollary}[theorem]{Corollary}

\newtheorem{proposition}[theorem]{Proposition}
\newtheorem{remark}[theorem]{Remark}

\newtheorem{conjecture}[theorem]{Conjecture}

\title{On a sequence of singular ball quotient surfaces on the line $K^2=9\chi-18$}
\author{Carlos Rito, Xavier Roulleau}
\date{}

\begin{document}
\maketitle

\begin{abstract}
Starting from computer experiments with the fundamental group of the Cartwright–Steger surface,
we construct an infinite tower \((X_n)_{n\ge1}\) of normal projective surfaces obtained
by successive \(\ZZ/3\)-Galois covers \mbox{\(X_{n} \to X_{n-1}\)}.
For $n>1$, their minimal resolutions \(\widetilde{X}_n\) lie on the line \(K^2 = 9\chi - 18\)
(equivalently \(c_1^2 = 3c_2 - 72\)), which is parallel to the Bogomolov--Miyaoka--Yau line
\(K^2 = 9\chi\) of ball quotients. We compute the fundamental groups for the first cases,
showing that \(\pi_1(\widetilde{X}_n)=1\) for $n=1,\ldots,5$.
Motivated by the geometry of the construction, we conjecture that all \(\widetilde{X}_n\) are simply connected.
\medskip

\noindent{Keywords:}
surfaces of general type; ball quotient surfaces; fundamental groups; simply connected surfaces.

\medskip

\noindent{2020 MSC:}
Primary~14J29; Secondary~14Q10,~20F05,~20F34.
\end{abstract}

\section{Introduction}

For smooth complex algebraic surfaces $S$ of general type, the Bogomolov--Miyaoka--Yau inequality $K^2\leq 9\chi$
(equivalently $c_1^2\leq 3c_2$) holds, where $K^2$ is the self-intersection of a canonical divisor of $S$
and $\chi$ is the holomorphic Euler characteristic of $S.$ Surfaces on the line $K^2=9\chi$ are known to be
quotients of the complex unit ball $\mathbb B$ by a lattice $\Lambda\subset PU(2,1).$
They are rigid surfaces with infinite fundamental group.
A long-standing problem in the geography of surfaces of general type is to determine whether there exists a
neighborhood of this line in which the fundamental group must be infinite and, in particular,
how closely the line can be approached by simply connected surfaces.
An important step has been taken by Roulleau and Urz\'ua \cite{RU},
who constructed a sequence of simply connected surfaces $Z_n$ with 
\[ \lim_{n}\frac{K_{Z_n}^2}{\chi(Z_n)}=9.\]
However, this is not optimal, because $\lim_{n} \left(9\chi(Z_n)-K_{Z_n}^2\right)=+\infty$,
thus a substantial region near the Bogomolov--Miyaoka--Yau line remains unexplored for the existence of
simply connected examples.

Surfaces of general type with invariants $K^2=9\chi=9$ and $p_g=0$ (hence $q=0$) are the so-called {\em fake projective planes}.
There are $50$ pairs of complex-conjugated such surfaces,
according to the results from the work of Prasad and Yeung \cite{PY1}, \cite{PY2},
and Cartwright and Steger \cite{CS}.

In their computational classification, Cartwright and Steger \cite{CS} also found the construction of a remarkable surface
$S_1$ with invariants
\[K_{S_1}^2=9,\ p_{g}(S_1)=q(S_1)=1.\]
This ball quotient surface is given by a lattice $\Lambda\subset PU(2,1)$ which is a sublattice of a maximal lattice
$\bar{\Gamma}$ constructed by Mostow, which is a Deligne-Mostow group whose associated
weights $(2,2,2,7,11)/12$ satisfy the condition $\Sigma(INT)$ (see \cite{Mostow}).
More recently, an explicit construction of this surface has been given by Borisov and Yeung \cite{BY}.

In this paper, we start from the Cartwright--Steger surface and,
guided by computer experiments with its fundamental group,
we uncover an infinite sequence $(X_n)$ of normal projective surfaces
lying on the line $K^2 = 9\chi - 18$ for $n > 1$.
Our computations show that the first five smooth minimal resolutions
\(\widetilde{X}_1,\dots,\widetilde{X}_5\) are simply connected. More
experiments then point to a simple geometric description: each
\(X_n\) arises as a Galois triple covering of \(X_{n-1}\), produced by a uniform procedure.
These findings motivate the central conjecture of the paper that all \(\widetilde{X}_n\) are simply
connected.

All computations were performed in Magma \cite{BCP}.
A reproducible script certifying our claims is available as an arXiv ancillary file.

\bigskip 
\noindent\textbf{Acknowledgments.}
The first author was financed by Portuguese Funds through FCT
(Funda\c c\~ao para a Ci\^encia e a Tecnologia) within the Project
UID/00013/ 2025: Centro de Matem\'atica da Universidade do Minho (CMAT/UM).
The second author was supported by the Centre Henri Lebesgue ANR-11-LABX-0020-01.

\section{Experiments with the Cartwright--Steger fundamental group}

Let \(\BB\subset\CC^2\) be the unit ball. Consider the
Cartwright--Steger ball quotient surface \(S_1=\BB/H_1\) and its $\ZZ/3$-quotient
\(X_1=\BB/G_1\), where \(H_1\) has index \(3\) in \(G_1\), and both
$H_1,G_1$ are subgroups of a maximal lattice $\bar\Gamma$.
The covering map \(S_1\to X_1\) is ramified at isolated points corresponding to
order-\(3\) elliptic elements in \(G_1\).

Starting from the finite presentation of \(\bar\Gamma\) given in \cite{CKYshort},
we searched, using Magma, for two elements that generate \(\bar\Gamma\) and yield short defining
relators after simplification. This led to the concise presentation
\[
\bar\Gamma \ \cong\ 
\Big\langle\, u,w \ \Big|\ u^3,\ w^3,\ (u,\ w u w^{-1} u w),\ (uw)^8 \,\Big\rangle,
\]
where \((x,y)=x^{-1}y^{-1}xy\).

We then searched for short words in \(u,w\)
that generate an index 288 subgroup isomorphic to the group \(G_1\le\bar\Gamma\) corresponding to \(X_1\).
We found that \(G_1\) is generated by the three conjugates
\[
G_1 \cong\ \big\langle\; w^{\,u w},\ \ w^{\,w^u},\ \ w^{\,(u,\,w^{-1})}\ \big\rangle
\ \le\ \bar\Gamma,
\]
where \(x^g=g^{-1}xg\).

The invariants of the abelianization of $G_1$ are $[3,3],$
i.e. $$\Abel(G_1):=G_1/[G_1,G_1]\cong (\ZZ/3)^2.$$

To find the group $H_1$ we compute all index-3 normal subgroups of $G_1.$
We get four subgroups. Their abelianizations have invariants $[0,0], [3,21], [3,3], [3,3]$.
The first one is $H_1.$

Let $G_2$ be one of the subgroups with invariants $[3,3]$, say the one containing the generator $w^{\,u w}.$
The computations show that $G_2$ also has four index-3 normal subgroups, whose abelianization invariants are
\begin{equation}\label{abinv}
[7,0,0],\ [3,3],\ [3,3],\ [3,3].
\end{equation}
We check computationally that the first one is an index-3 subgroup of $H_1$,
while the other three are conjugate subgroups.

Iterating this procedure we get a chain of nine successive index-3 subgroups
$$G_1\geq G_2\geq\cdots \geq G_{10},$$
at each step exhibiting the same subgroup pattern as in \eqref{abinv}.
For each $G_i$ there is an index-9 normal subgroup $H_{i+1}$ such that
$$G_i/H_{i+1}\cong(\ZZ/3)^2.$$
We check that $$H_{i+1}=[G_i,G_i],\ \ i=1,\ldots 9.$$

Due to computer limitations, our computations do not go further,
and we ask whether this index-3 subgroup pattern persists indefinitely.

\[
\begin{tikzcd}[row sep=1.0em, column sep=3em]
\vdots \arrow[r, phantom] \arrow[d] & \vdots \arrow[d] \\
H_{2}  \arrow[r] \arrow[d] & G_{2}  \arrow[d] \\
H_{1}  \arrow[r]            & G_{1}
\end{tikzcd}
\]

To answer this, we looked for a much simpler group that still exhibits the
same index-\(3\) subgroup pattern. A short randomized search among
two–generator, small–relator groups pointed to the Euclidean triangle group
\[
T \;=\; \langle\, x,y \mid x^{3},\ y^{3},\ (xy)^{3}\,\rangle .
\]
The group $T$ contains an index $3$ subgroup isomorphic to $T$, thus it
has an infinite structure of index-3 subgroups similar to the one above,
with abelianizations
$$[0,0],\ [3,3],\ [3,3],\ [3,3].$$

We then used Magma to find a surjective homomorphism $$h:G_1\twoheadrightarrow T,$$
which then implies that $G_1$ also has an infinite structure of index-3 subgroups as above.

From the map $h$ we extracted order-3 generators $a,b,c,d$ of $G_1$ such that
$$h(a)=h(b)=h(c)=x \in T,\ \ h(d)=y\in T.$$
In particular, $ab^{-1}$ and $bc^{-1}$ are in the kernel of $h.$

We wonder if $\Ker(h)$ is normally generated by these two elements.
Accordingly, set
$$Q\ :=\ G_1\ \big/\ \big\langle\!\big\langle\, ab^{-1},\,bc^{-1}\,\big\rangle\!\big\rangle^{G_1}.$$
A Magma computation shows that
$$Q\ \cong\ T'\ :=\ \big\langle\, x,y\ \big|\ x^{3},\ y^{3},\ (xy)^{3}(yx)^{3}\,\big\rangle,$$
which still exhibits the same index-3 normal-subgroup pattern as $T$.

So, we have seen that there is an infinite commutative diagram
\[
\begin{tikzcd}[row sep=1.0em, column sep=3em]
\vdots \arrow[r, phantom] \arrow[d] & \vdots \arrow[d] \\
S_{2}  \arrow[r] \arrow[d] & X_{2}  \arrow[d] \\
S_{1}  \arrow[r]            & X_{1}
\end{tikzcd}
\]
where the arrows denote $\ZZ/3$-coverings of surfaces.

The group $G_1$ is generated by elements of order 3,
hence by Armstrong theorem \cite{Armstrong},
the smooth minimal model $\tilde X_1$ of the surface $X_1$ is simply connected.
What can be said about $X_2, X_3,\ldots$?

Using Magma, we show that
$$G_{i+1}\ :=\ \big\langle\!\big\langle\, a,b,c\,\big\rangle\!\big\rangle^{G_i},\ \ i=1,2,3$$
(the normal closure of $\big\langle\, a,b,c\,\big\rangle$ in $G_i$).
Beyond this point, the presentations grow rapidly in length and complexity, and our computations stall.

To handle $G_5$ we use a workaround:
we compute the quotient
$$J\ :=\ G_5\ \big/\ \big\langle\!\big\langle\, a,b,c\,\big\rangle\!\big\rangle^{G_4}$$
and use the Magma function \texttt{SearchForIsomorphism} to check if it is isomorphic to $\ZZ/3.$
In the process, Magma reports (twice) that a generator of $J$ is trivial.
We append this relation and the algorithm confirms that $J\ \cong\ \ZZ/3.$

Summing up, the groups $G_1,\ldots,G_5$ are generated by elements of order 3,
which imply that the smooth minimal surfaces $\tilde X_1,\ldots,\tilde X_5$ are simply connected.

\section{The fibration on the surface $X_1$}\label{CSsurface}

Our references for this section are \cite{CS} and \cite{CKYshort}.

The Cartwright-Steger surface is $S_1:=\mathbb B/H_1$, with invariants 
$K^2=9\chi=9,$ $p_{g}=q=1.$
The group $G_1$ is the normalizer of $H_1$ in $\bar\Gamma.$
The automorphism group $\sigma:=\Aut(S_1)$ is $\ZZ/3\cong G_1/H_1.$
The quotient surface $X_1:=S_1/\sigma$ has $9$ singularities:
$3$ points $O_{1}',O_{2}',O_{3}'$ of type $\frac{1}{3}(1,1),$
and $6$ points $Q_{1}',\ldots,Q_{6}'$ of type $\frac{1}{3}(1,2)$
($A_{2}$ singularities, or ordinary {\em cusps}).

Denote by $O_i, Q_j\in S_1$ the fixed points of $\sigma$ corresponding to $O_i',$ $Q_j',$ respectively.
The Albanese fibration $\alpha:S_1\to E$ has three fibers $\bar F_1,$ $\bar F_2,$ $\bar F_3$
such that $O_1,O_2,O_3\in\bar F_1,$ $Q_1,Q_2,Q_3\in\bar F_2$ and $Q_4,Q_5,Q_6\in\bar F_3.$
Moreover, $\bar F_1,$ $\bar F_2,$ $\bar F_3$ are reduced.
The invariants of the smooth minimal resolution $\tilde X_{1}$ of $X_1$ are
\[K_{\tilde X_{1}}^2=\chi(\tilde X_{1})=2,\ p_g(\tilde X_{1})=1,\ q(\tilde X_{1})=0.\]

By the functorial properties of the Albanese map, the automorphism
$\sigma$ acts on the elliptic curve $E$ through an automorphism
$\sigma_{E}$, which has $3$ fixed points, and $\alpha$ induces a fibration
\[f:\tilde X_{1}\longrightarrow\PP^{1}=E/\sigma_{E}.\]

Let $F_{i}$ and $D_i$ be the total and strict transform in $\tilde X_1,$ respectively,
of the image of $\bar F_i$ on $X_1:=S_1/\sigma.$

\begin{proposition}
\label{thm:The-fibration-}
The fibers $F_{1},F_{2},F_{3}$ have the following configuration:
\begin{itemize}
\item $F_{1}=3D_{1}+B_{1}+B_{2}+B_{3},$
\item $F_{2}=3D_{2}+2(A_{1}+A_{2}+A_{3})+A_{1}'+A_{2}'+A_{3}',$
\item $F_{3}=3D_{3}+2(A_4+A_5+A_6)+A_4'+A_5'+A_6',$
\end{itemize}
where $B_i$ is the $(-3)$-curve which resolves the singularity $O_i',$ $i=1,2,3,$ and
$A_i,$ $A_i'$ are the $(-2)$-curves which resolve the singularity $Q_i',$ $i=1,\ldots,6.$
\end{proposition}

\begin{proof}
We have a commutative diagram
\[
\begin{CD}
\tilde S_1     @>\pi>>  S_1\\
@V\psi VV        @VV\sigma V\\
\tilde X_1     @>\varphi>>  X_1
\end{CD}
\]
where $\varphi$ is the minimal resolution of the singularities of $X_1,$ the map $\psi$ is a
$\ZZ/3$-covering ramified over the exceptional divisors from $\varphi,$ and
$\pi$ is a sequence of blowups centered at the fixed points of $\sigma.$
Let $\tilde F_i:=\pi^*(\bar F_i),$ $i=1,2,3.$ The pullback of a fiber in $\tilde X_1$ is a union of three fibers
in $\tilde S_1,$ except for $\psi(F_i)=3\tilde F_i,$ $i=1,2,3.$ This implies that every component of a fiber $F_i$ which
is not contained in the branch locus of $\psi$ must be of multiplicity $3.$
We have then
\begin{align*}
F_1&=3D_1+\textstyle\sum_1^3 a_{1i}B_i,\\
F_2&=3D_2+\textstyle\sum_1^3(a_{2i}A_i+b_{2i}A_i'),\\
F_3&=3D_3+\textstyle\sum_4^6(a_{3i}A_i+b_{3i}A_i'),
\end{align*}
for some effective divisors $D_1,D_2,D_3,$ and $a_{ji},b_{ji}\in\mathbb N.$

According to \cite[Section~5]{R}, the fibers of the Albanese map of $S_1$ are smooth.
This implies $B_iD_1=1,$ $i=1,2,3,$ and $A_iD_2=1,$ $A_i'D_2=0,$ $i=1,2,3$ (possibly relabeling $A_i\leftrightarrow A_i'$).
Now $F_1B_i=F_2A_i=F_2A_i'=0,$ for $i=1,2,3,$ gives $a_{1i}=1,$ $a_{2i}=2,$ $b_{2i}=1,$ for $i=1,2,3.$
The configuration of $F_3$ is analogous to the one of $F_2.$

\end{proof}

\section{Geometric construction of the tower}\label{sec:Geometrie}

\subsection{Basics on Galois triple covers}
Our reference here is \cite{Tan1}.

Let $\tilde X$ be a smooth surface. A Galois triple cover $\pi:\tilde Y\rightarrow \tilde X$ is determined by
divisors $L,$ $M,$ $B$, $C$ on $\tilde X$ such that $B\in |2L-M|$ and $C\in |2M-L|$.
The branch locus of $\pi$ is $B+C$ and $3L\equiv 2B+C,$ $3M\equiv B+2C$
(we say that $2B+C$ and $B+2C$ are {\em $3$-divisible}). The surface $\tilde Y$
is normal iff $B+C$ is reduced. The singularities of $\tilde Y$ lie over the singularities of $B+C$.

Now suppose that $\tau:\tilde X\rightarrow X$ is the minimal resolution of a normal surface $X$ with a set
$\{s_1,\ldots,s_{\bar n}\}$ of ordinary cusps, a set $\{q_1,\ldots,q_{\bar m}\}$ of singularities of type $\frac{1}{3}(1,1),$
and no other singularities.
Let $B_i:=\tau^{-1}(q_i)$ be the $(-3)$-curve which resolves $q_i,$ $i=1,\ldots,\bar m.$
If the $(-2)$-curves $A_i,$ $A_i'$ satisfying $\tau^{-1}(s_i)=A_i+A_i'$ can be
relabeled such that
\[\sum_1^{\bar n} (2A_i+A_i')+\sum_1^{\bar m} B_i\equiv 3J,\]
for some divisor $J,$ then we say that the singular set of $X$ is $3$-divisible.
We have a commutative diagram
\[
\begin{array}{ccc}
\tilde Y & \longrightarrow & Y\\
\downarrow &  & \downarrow\\
\tilde X & \longrightarrow & X
\end{array}
\]
where the vertical arrows over $X$ and $\tilde X$ are Galois triple covers ramified over the singular set of $X$ and
over its resolution, respectively. The surface $Y$ is smooth and $\tilde Y\to Y$ is a sequence of blowups.

\begin{proposition}\label{invariants}
If\ $\bar n=3n$ and\ $\bar m=3m,$ with $n,m\in\mathbb N,$ then
\[
\chi(Y)=3\chi(\tilde X)-2n-m, \ \ \ K_{Y}^{2}=3K_{\tilde X}^{2}+3m.
\]
\end{proposition}

\begin{proof}
If $m=0,$ this is just \cite[Lemma 2.2.4]{Tan1}.
The contribution of $m\ne 0$ follows easily from \cite[Section 1.3]{Tan1}.
\end{proof}

\subsection{The construction \label{sec:Geometrie}}

Recall that the resolution $\tilde X_1$ of the surface $X_1:=S_1/\sigma$ has a fibration onto $\mathbb P^1$
with singular fibers $F_1,$ $F_2,$ $F_3.$ The $3$-divisibility of the divisors
\[
F_1+2F_2,\ F_1+2F_3,\ F_2+2F_3,\ F_1+F_2+F_3
\]
implies that the following sets of singularities of $X_1$ are $3$-divisible:
\begin{align*}
\Delta_1&=\{O_1',\, O_2',\, O_3',\, Q_1',\, Q_2',\, Q_3'\},\\
\Delta_2&=\{O_1',\, O_2',\, O_3',\, Q_4',\, Q_5',\, Q_6'\},\\
\Delta_3&=\{Q_1', \ldots, Q_6'\},\\
\Delta_4&=\{O_1',\, O_2',\, O_3',\, Q_1', \ldots, Q_6'\}.
\end{align*}
To each $\Delta_i$ corresponds a $\ZZ/3$-covering $X_1(i)\to X_1.$
One has $X_1(4)=S_1,$ the Cartwright-Steger surface. 
The singular set of each $X_1(1), X_1(2)$ is a union of $9$ ordinary cusps, and
the singular set of $X_1(3)$ is a union of $9$ singularities of type $\frac{1}{3}(1,1).$

We now fix $X_2:=X_1(2)$ (or $X_1(1),$ the resulting surfaces have the same invariants).
The resolution $\tilde X_2$ of $X_2$ has three singular fibers $F_1',$ $F_2',$ $F_3'$ which are copies of $F_2.$
Equivalently $X_2$ has three multiple fibers, each containing three cusp singularities of $X_2.$
As above, we can construct three
$\ZZ/3$-coverings of $X_{2},$ each ramified on $6$ cusps, using the $3$-divisible divisors
\[
F_1'+2F_2',\ F_1'+2F_3',\ F_2'+2F_3'.
\]
Fixing any of these, we obtain a new surface $X_{3}$
which has the same special fibers as $X_{2}.$ Repeating this process,
we construct a sequence $(X_{n})$ of surfaces each
containing three special fibers with the same configuration as above.

Let $G_{n+1}$ be the subgroup of $PU(2,1)$ such that $X_{n+1}=\BB/G_{n+1}.$
We define
\[
H_{n+1}:=H_{n}\cap G_{n+1}\ \ \ \ {\rm and}\ \ \ \ S_{n+1}:=\BB/H_{n+1}.
\]
(Notice that it is well defined for all $n\geq 1,$ the group $H_1$ is the fundamental group of the Cartwright-Steger surface $S_1.$)
Since $H_{n}$ and $G_{n+1}$ are index $3$ normal subgroups of $G_n,$ then $H_{n+1}$ is an index $9$ normal subgroup of $G_n.$
We are therefore considering the following diagram
\begin{equation}\label{eq:diagram galois cover}
\begin{array}{ccc}
S_{n+1} & \longrightarrow & X_{n+1}\\
\downarrow &  & \downarrow\\
S_{n} & \longrightarrow & X_{n}
\end{array}
\end{equation}
such that $S_{n+1}\to X_n$ is a $(\ZZ/3)^2$-Galois covering.

We note that each group $H_n$ is torsion free, because $H_1$ is torsion free, hence
$S_n$ is smooth and then the map $S_n\to X_n$ is ramified over the $9$ cusps of $X_n,$ for $n>0.$
Since $(\ZZ/3)^2$ has four $\ZZ/3$ subgroups, the covering $S_{n+1}\to X_n$ has four intermediate surfaces
\[
X_{n+1},\ X_{n+1}',\ X_{n+1}'',\ S_n,
\]
that correspond to index $3$ subgroups
\[
\ G_{n+1},\ G_{n+1}',\ G_{n+1}'',\ H_{n}
\]
of $G_n.$
Clearly the intersection of any two of these groups is $H_{n+1}.$
We have inclusions of normal index $3$ groups:
\[
\begin{array}{ccc}
H_{n+1} & \longrightarrow & G_{n+1}\\
\downarrow &  & \downarrow\\
H_{n} & \longrightarrow & G_{n}
\end{array}.
\]

Let $\tilde X_n\to X_n$ be the smooth minimal resolution of $X_n.$ One has:
\begin{proposition}\label{invs}
The surface $\tilde X_{n}$ is minimal, and for $n>1:$
\begin{itemize}
\item[$a)$] $K_{\tilde X_n}^2=3^{n},\ \chi(\tilde X_n)=3^{n-2}+2;$
\item[$b)$] $K_{S_n}^2=3^{n+1},\ \chi(S_n)=3^{n-1}.$
\end{itemize}
\end{proposition}

\begin{proof}
We have $K_{\tilde X_1}^2=\chi(\tilde X_1)=2.$ Then Proposition \ref{invariants} with $n=m=1$
gives that $K_{\tilde X_2}^2=9,$ $\chi(\tilde X_2)=3.$ Now Proposition \ref{invariants} with $n=2,$ $m=0$
gives $a).$ Applying Proposition \ref{invariants} to the covering $S_n\to X_n$ ($n=3,$ $m=0$),
we get $b).$
\end{proof}

\begin{corollary}
The surfaces $S_n$ are smooth ball quotients. The surfaces $\tilde X_n$ are on the line
$K^2=9\chi-18$ for $n>1$. In particular, $\lim K_{\tilde X_n}^2/\chi(\tilde X_n)=9.$
\end{corollary}

\section{Conclusion}

Starting from the Cartwright-Steger surface $S_1$ and its $\ZZ/3$-quotient $X_1$,
we have constructed two sequences of surfaces $(S_{n}),$
$(X_{n})$ such that the following diagram
\[
\begin{array}{ccc}
S_{n+1} & \longrightarrow & X_{n+1}\\
\downarrow &  & \downarrow\\
S_{n} & \longrightarrow & X_{n}
\end{array}
\]
fits in a $(\ZZ/3)^2$-Galois covering $S_{n+1}\to X_n.$
Moreover, for $n>1$ the invariants of the smooth surfaces $\tilde X_n$ are on the line $$K^2=9\chi-18,$$
and the surfaces $\tilde X_{1},\ldots \tilde X_{5}$ are simply connected. 

Motivated by the simplicity of the geometric construction of the \(\ZZ/3\)-covers
\(X_n \to X_{n-1}\), we formulate:
\begin{conjecture}
All surfaces $\tilde X_{n}$ are simply connected.
\end{conjecture}

\begin{remark}
In Section \textup{\ref{sec:Geometrie}}, we have chosen to study the tower of surfaces
obtained by branching over cusps. If instead we choose singularities
$\frac{1}{3}(1,1)$, we obtain a sequence of surfaces $(W_{n})$
with invariants on the line $K^2=9\chi-12$.
Moreover, one can compute that the fundamental group of the first few
surfaces $W_{i}$ is $\ZZ/7$. 
\end{remark}

\bibliography{References}

\vspace{1cm}

\noindent Carlos Rito
\vspace{0.1cm}
\\ Centro de Matem\'atica, Universidade do Minho - Polo CMAT-UTAD
\vspace{0.1cm}
\\ Universidade de Tr\'as-os-Montes e Alto Douro, UTAD
\\ Quinta de Prados
\\ 5000-801 Vila Real, Portugal
\vspace{0.1cm}
\\ www.utad.pt, {\tt crito@utad.pt}

\vspace{1cm}

\noindent Xavier Roulleau \vspace{0.1cm} 
\\Université d’Angers,
\\CNRS, LAREMA, SFR MATHSTIC,
\\F-49000 Angers, France
\\ \verb|xavier.roulleau@univ-angers.fr|

\end{document}